\documentclass[12pt]{amsart}
\usepackage{amsmath}
\usepackage{amssymb}
\usepackage{amsthm}
\usepackage[margin=1.5in]{geometry}
\usepackage{mathrsfs}
\usepackage{amsthm}
\newcommand{\ZZ}{\mathbb{Z}}
\newcommand{\RR}{\mathbb{R}}
\newcommand{\EE}{\mathbb{E}}
\newcommand{\NN}{\mathbb{N}}

\newcommand{\N}{\mathcal{N}}

\usepackage{hyperref}
\hypersetup{
    colorlinks=true,
    linkcolor=blue,
    citecolor=red,
    filecolor=magenta, 
    urlcolor=cyan
}

\newtheorem{theorem}{Theorem}[section]
\newtheorem{corollary}[theorem]{Corollary}
\newtheorem{proposition}[theorem]{Proposition}
\newtheorem{lemma}[theorem]{Lemma}
\theoremstyle{definition}
\newtheorem{remark}[theorem]{Remark}
\newtheorem{definition}[theorem]{Definition}

\begin{document}

\title{Lattice points on a curve via $\ell^2$ decoupling}

\author{Daishi Kiyohara}

\address{Department of Mathematics, Harvard University, 1 Oxford Street\\
Cambridge, MA 02138, United States}
\email{dkiyohara@math.harvard.edu}

\maketitle

\begin{abstract}
This paper extends Bombieri and Pila's estimate of lattice points on curves to arbitrary finite sets by incorporating considerations of minimal separation and the doubling constant.
We derive the estimate by establishing the $\ell^2$ decoupling inequality for non-degenerate curves in $\RR^n$.
Additionally, we review the curve-lifting method introduced in Bombieri and Pila's work and establish the estimate of lattice points in the neighborhood of a planar curve.
\end{abstract}



\section{Introduction}
In his pioneering work \cite{J}, Jarn\'ik proved that a strictly convex planar curve of length $\ell$ contains at most $3(4\pi)^{-1/3}\ell^{2/3}+O(\ell^{1/3})$ integer points.
This inequality can be reformulated by fixing a strictly convex compact curve $\Gamma\subset \RR^2$ and examining the set $\Gamma\cap (\frac{1}{N}\ZZ)^2$ of $\frac{1}{N}$-integer points on the curve.
Then Jarn\'ik's theorem implies the upper bound $|\Gamma\cap (\frac{1}{N}\ZZ)^2|= c(\Gamma)N^{\frac{2}{3}}+O(N^{1/3})$ where the constant $c(\Gamma)$ only depends on the length of a curve.

Since Jarn\'ik's work, there have been many advancements strengthening this inequality for planar curves by imposing additional smoothness conditions.
For example, Swinnerton-Dyer \cite{Sw} proved the upper bound $|\Gamma\cap (\frac{1}{N}\ZZ)^2|\lesssim N^{\frac{3}{5}+\epsilon}$ for a strictly convex $C^3$ graph $\Gamma:(x,f(x))$  where the implicit constant depends on the graph and $\epsilon$.
Schmidt refined this result, conjecturing an improved exponent of $\frac{1}{2}$, under the assumption that $f''$ is weakly monotonic and vanishes at most once.
The parabola $\Gamma= \{(x,x^2): x\in [0,1]\}$ demonstrates that the exponent $\frac{1}{2}$ is optimal.

In the remarkable paper \cite{BP}, Bombieri and Pila established new upper bounds under the assumption that the curve is $C^n$ and satisfies the finiteness condition regarding intersections with algebraic curves.
Specifically, \cite[Theorem 1]{BP} implies that for a planar graph $\Gamma:(x,f(x))$ which is $C^{\frac{1}{2}(s+1)(s+2)-1}$ and has uniformly bounded intersection of size $\le S$ with algebraic curves of degree at most $s$ the inequality
\[|\Gamma\cap (\tfrac{1}{N}\ZZ)^2|\le C(\Gamma)N^{\frac{8}{3(s+3)}}\]
holds, where $C(\Gamma)$ depends on $S$ and $\|f\|_{C^{\frac{1}{2}(s+1)(s+2)-1}}$.
Furthermore, they prove the inequality $|\Gamma\cap (\frac{1}{N}\ZZ)^2|\le c(\Gamma,\epsilon) N^{\frac{1}{2}+\epsilon}$ for any $C^\infty$ strictly convex curve $\Gamma$.

Their proof employs a curve-lifting method, implicitly considering a certain lift of the planar curve into a higher-dimensional space, which we review later.
The estimate of lattice points on a planar curve is derived from the inequality
\[|\Gamma\cap (\tfrac{1}{N_1}\ZZ\times\cdots\times\tfrac{1}{N_n}\ZZ)|\le C(\Gamma) (N_1\cdots N_n)^{\frac{2}{n(n+1)}}\]
for any $C^n$ curve $\Gamma\subset\RR^n$ with uniformly bounded intersections with hyperplanes.

The purpose of this paper is to establish a variant of Bombieri-Pila's inequality using the recently developed theory of $\ell^2$ decoupling, especially the sharp $\ell^2$ decoupling inequality for the moment curve in \cite{BDGuth}.
With this approach, we extend the estimate from lattices to arbitrary finite sets.
\subsection{Main theorems}
To state our main results, we introduce necessary terminology.
Consider a compact $C^{n,\alpha}$ curve $\Gamma$ in $\RR^n$, parameterized by
\[\gamma(t)=(\gamma_1(t),\cdots,\gamma_n(t))\]
for $t\in[0,1]$.
We say $\Gamma$ is \textit{non-degenerate} if the Wronskian
\[W(\gamma_1',\cdots,\gamma_n')(t)
=\begin{vmatrix} \gamma_1'(t)&\gamma_2'(t)&\cdots&\gamma_n'(t)\\
\gamma_1''(t)&\gamma_2''(t)&\cdots&\gamma_n''(t) \\
\vdots&\vdots&\ddots&\vdots\\
\gamma_1^{(n)}(t)&\gamma_2^{(n)}(t)&\cdots&\gamma_n^{(n)}(t)
\end{vmatrix}\]
is nonvanishing.
For a finite set $A$, we define its doubling constant $K$ by $K=\frac{|A+A|}{|A|}$.

Our main theorem is as follows.
\begin{theorem}\label{thm:main}
Let $\Gamma$ be a $C^{n,\alpha}$ non-degenerate curve inside $\RR^n$ for positive $\alpha$ with $W(\Gamma)(t)>c_0>0$ for all $t\in [0,1]$.
Let $A$ be a finite subset of $\RR^n$ with doubling constant $K$ and minimal separation $s$.
Then we have
\[|\N_{s^n}(\Gamma)\cap A|\lesssim s^{-\epsilon}K\cdot |A|^{\frac{2}{n(n+1)}}\]
where $\N_{\delta}(\Gamma)$ denotes the $\delta$-neighborhood of $\Gamma$.
Here the implicit constant depends on $\epsilon$, $\|\gamma\|_{C^{n,\alpha}}$ and $c_0$.
\end{theorem}

The theorem provides a generalization of Bombieri-Pila's inequality for a curve inside $\RR^n$ from lattice points to arbitrary finite sets.
We note that the inequality now involves the doubling constant and the minimal separation.

By combining Theorem \ref{thm:main} with the curve-lifting method, we establish the following estimate of lattice points on a planar curve.
See section 4 for the notation.

\begin{theorem}\label{thm:lat}
Let $\Gamma$ be a compact $C^{n,\alpha}$ planar curve.
Suppose that there is a collection $\mathcal{M}=\{m_1,\cdots,m_n\}$ of distinct monomials including $x$ and $y$ such that $W^{\mathcal{M}}(\Gamma)(t)>c_0>0$ for all $t\in [0,1]$.
Then we have
\[|\N_{d/N^{n}}(\Gamma)\cap (\tfrac{1}{N}\ZZ)^2|\lesssim N^{e(\mathcal{M})+\epsilon}\]
for each $\epsilon>0$, where the exponent $e(\mathcal{M})$ is given by 
\[e(\mathcal{M})=\frac{2}{n(n+1)}\sum_{i=1}^n \deg m_i.\]
Moreover the implicit constant depends on $\epsilon$, $\|\gamma\|_{C^{n,\alpha}}$, $c_0$ and $d$.  
\end{theorem}

\vspace{6pt}

The paper is structured as follows.
In Section 2 we review the theory of $\ell^2$ decoupling, especially the sharp $\ell^2$ decoupling for the moment curve.
Then we derive the $\ell^2$ decoupling inequality for non-degenerate curves and combined it with additive combinatorics to prove Theorem \ref{thm:main}.
Section 3 reviews the property of curve-lifting, implicitly used in Bombieri and Pila's work \cite{BP}.
Section 4 establishes Theorem \ref{thm:lat}, providing the estimate of lattice points in the neighborhood of a planar curve, by combining Theorem \ref{thm:main} and the curve-lifting method.
Additionally, we prove that the non-degeneracy of a curve implies the finiteness condition about intersections with hyperplanes used in \cite{BP}.

\subsection*{Notation}
Throughout the paper, we write $A\lesssim B$ to denote $A\le CB$ for an implicit constant which depends on parameters like $\epsilon$.
\vspace{12pt}

\section{$\ell^2$ decoupling for non-degenerate curves and application to lattice points}
In this section, we first review the theory of $\ell^2$ decoupling and derive the $\ell^2$ decoupling inequality for a non-degenerate $C^{n,\alpha}$ curve inside $\RR^n$, building upon the results for the moment curve established in \cite{BDGuth}.
The subsequent subsection is devoted to proving Theorem \ref{thm:main} using the $\ell^2$ decoupling inequality.
\subsection{$\ell^2$ decoupling for a non-degenerate curve}
For a function $f:\RR^n\to\mathbb{C}$ and a subset $R\subset \RR^n$, the Fourier restriction $f_R$ is defined by
\[f_R(x)=\int_R\hat{f}(\xi)e^{2\pi i x\cdot\xi}d\xi.\]
Consider a subset $\Omega$ of $\RR^n$ partitioned into disjoint subsets $\theta$.
Any function $f$ with $\hat{f}$ supported on $\Omega$ can be decomposed into its Fourier restrictions as $f=\sum_\theta f_\theta$.
The $\ell^2$ decoupling inequality examines the relationship between the $L^p$ norm of $f$ and the $L^p$ norms of $f_\theta$.
\begin{definition}
For a subset $\Omega$ of $\RR^n$ partitioned into disjoint subsets $\theta$, and for every exponent $p$, the \textit{decoupling constant} $D_p(\Omega =\sqcup \theta)$ is defined as the smallest constant such that, for every function $f$ with $\hat{f}$ supported on $\Omega$, the following holds:
\[\|f\|_{L^p(\RR^n)}\le D_p(\Omega =\sqcup \theta)\left(\sum_{\theta}\|f_{\theta}\|_{L^p(\RR^n)}^2\right)^{\frac{1}{2}}.\]
\end{definition}

We begin by reviewing the $\ell^2$ decoupling inequality for the moment curve, a key element in our approach.
Let $M$ be the moment curve in $\RR^n$ parameterized by $(t,t^2,\cdots,t^n)$ for $t\in[0,1]$, and $\N_{\delta}(M)$ be the $\delta$-neighborhood of the moment curve $M$.
The partition of $\N_{\delta}(M)$ into disjoint subsets $\theta$, each containing the $\delta$-neighborhood of an arc of $M$ of length $\delta^{\frac{1}{n}}$, defines the decoupling constant $D_{p}(\delta)$ as the smallest constant such that
\[\|f\|_{L^p(\RR^n)}\le D_{p}(\delta)\left(\sum_{\theta}\|f_{\theta}\|_{L^p(\omega_B)}^2\right)^{\frac{1}{2}}\]
holds for all functions $f$ with $\hat{f}$ supported in $\N_{\delta}(M)$.
For a precise construction of the partition $\N_\delta(\Gamma)=\sqcup\theta$ we refer to \cite{GLYZ}.

Bourgain, Demeter and Guth established the following inequality for the decoupling constant $D_{p}(\delta)$ at the critical value $p=n(n+1)$, with the original proof in \cite{BDGuth} and a concise proof in \cite{GLYZ}.
\begin{theorem}\label{thm:bdg}
For every $\delta\in(0,1)$, we have $D_{n(n+1)}(\delta) \lesssim \delta^{-\epsilon}$.
\end{theorem}

Now consider a non-degenerate $C^{n,\alpha}$ curve $\Gamma$ inside $\RR^n$.
In this case, we similarly define the decoupling constant $D_{\Gamma,p}(\delta)$ for $\Gamma$.
Let $\N_{\delta}(\Gamma)$ be the $\delta$-neighborhood of $\Gamma$ and consider its partition into its disjoint subsets $\theta$ each containing the $\delta$-neighborhood of an arc of length $\delta^{\frac{1}{n}}$.
Then we define $D_{\Gamma,p}(\delta)$ to be the decoupling constant with respect to the partition $\N_{\delta}(\Gamma)=\sqcup \theta$.
The following theorem establishes that the decoupling inequality extends to $D_{\Gamma,p}(\delta)$ (see \cite{GLYZ} for the proof when the curve is $C^{n+1}$).
For the proof, we use Pramanik-Seeger's argument in \cite{PS} that allows us to generalize the $\ell^2$ decoupling inequality from a curve to its small perturbation.
\begin{theorem}\label{thm:dc}
Let $\Gamma$ be a non-degenerate $C^{n,\alpha}$ curve lying inside $\RR^n$ such that $W(\gamma_1',\cdots,\gamma_n')(t)> c_0>0$ throughout $[0,1]$.
Then we have
\[D_{\Gamma,n(n+1)}(\delta)\lesssim \delta^{-\epsilon}\]
where the implicit constant depends on $\epsilon$, $\|\gamma\|_{C^{n,\alpha}}$ and $c_0$.
\end{theorem}

\begin{proof}
Let $p=n(n+1)$.
For a given $t_0\in[0,1]$, consider the $n$-th Taylor series $r(t)$ of $\gamma(t)$ at $t_0$.
The $C^{n,\alpha}$ condition implies $|\gamma(t_0+\Delta t)-r(t_0+\Delta t)|\lesssim |\Delta t|^{n+\alpha}<\delta$ for $|\Delta t|\lesssim \delta^{\frac{1}{n+\alpha}}$, where the implicit constant only depends on $\|\gamma\|_{C^{n,\alpha}}$.
Let $R$ be the curve parameterized by $r(t)$ and denote by $\tau$ the $\delta$-neighborhood of an arc of $\Gamma$ of length $\delta^{\frac{1}{n+\alpha}}$.

The condition $W(\gamma_1',\cdots,\gamma_n')(t)>c_0>0$ implies that $R$ can be transformed into to the moment curve $M$ under an invertible linear change of variables.
Therefore, we obtain the inequality
\begin{align*}
\|f\|_{L^p}^2
&\le D_{\Gamma,p}^2(\delta^{\frac{n}{n+\alpha}})\cdot\sum_{\tau}\|f_{\tau}\|_{L^p}^2\\
&\lesssim  \delta^{-\epsilon}D_{\Gamma,p}^2(\delta^{\frac{n}{n+\alpha}}) \cdot\sum_{\tau}\sum_{\theta\subset\tau}\|f_{\theta}\|_{L^p}^2
\end{align*}
where we used Theorem \ref{thm:bdg}.
By the definition of the decoupling constant we have
\begin{align*}
D_{\Gamma,p}^2(\delta)\lesssim \delta^{-\epsilon}D_{\Gamma,p}^2(\delta^{\frac{n}{n+\alpha}}).
\end{align*}
Through iteration, we conclude that $D_{\Gamma,p}(\delta)\lesssim \delta^{-\epsilon}$.
\end{proof}

\subsection{Proof of the main theorem}
For a finite set $A$, the \textit{additive $m$-energy} of $A$ is defined by 
\[\EE_m(A)=|\{(a_1,\cdots,a_{2m})\in A^{2m}:a_1+\cdots+a_m=a_{m+1}+\cdots+a_{2m}\}|.\]
Recall that the doubling constant $K$ of $A$ is defined by $K=\frac{|A+A|}{|A|}$.
For the proof of the main theorem, we will need the following results from additive combinatoics.
\begin{lemma}\label{lem:aux}
For a finite subset $B$ of an abelian group, we have
\[\EE_m(B)\ge\frac{|B|^{2m}}{|mB|}\]
where $mB=\underbrace{B+\cdots+B}_{m \textit{ times}}$.
\end{lemma}
\begin{proof}
With the notation $\varphi(b)=|\{(b_1,\cdots,b_m)\in B^m:b_1+\cdots+b_m=b\}|$, the additive $m$-energy of $B$ can be expressed as follows
\[\EE_m(B) = \sum_{b\in mB}\varphi(b)^2.\]
Now the application of the Cauchy-Schwartz inequality implies
\[\EE_m(B) =\sum_{b\in mB}\varphi(b)^2\ge \frac{1}{|mB|}\left(\sum_{b\in mB}\varphi(b)\right)^2=\frac{|B|^{2m}}{|mB|}.\]
\end{proof}
\begin{lemma}\label{lem:db}
Let $B$ be a finite subset of an abelian group and let $A$ be another finite set with doubling constant $K$ containing $B$.
Then
\[\EE_m(B)\ge\frac{1}{K^m}\frac{|B|^{2m}}{|A|}.\]
\end{lemma}
\begin{proof}
Pl\"{u}nnecke's inequality asserts that $|mA|\le K^m|A|$ for a finite subset $A$ with doubling constant $K$.
The statement follows from Lemma \ref{lem:aux} by noting $|mB|\le|mA|$.
\end{proof}

\begin{proof}[Proof of Theorem \ref{thm:main}]
Let $B=N_{s^n}(\Gamma)\cap A$.
From Theorem \ref{thm:dc}, $D_{\Gamma,n(n+1)}(\delta)\lesssim \delta^{-\epsilon}$.
As a corollary of the $\ell^2$ decoupling inequality, we have\[\EE_m(B)\lesssim s^{-\epsilon} |B|^m\]
where $m=\frac{n(n+1)}{2}$ (see \cite[Section 4]{BDGuth} or \cite[Theorem 2.18]{BD}).
Note that this is sharp up to $s^{-\epsilon}$ by the obvious lower bound $\EE_m(B)\ge |B|^m$.
On the other hand, Lemma \ref{lem:db} provides a lower bound for the additive $m$-energy
\[\frac{1}{K^m}\frac{|B|^{2m}}{|A|}\le \EE_m(B).\]
Combining these two inequalities implies the desired inequality
\[|B|\lesssim s^{-\epsilon}K|A|^{1/m}.\]
\end{proof}
\subsection{Special case: generalized arithmetic progressions}
When the doubling constant is uniformly bounded, Theorem \ref{thm:main} yields an inequality without reference to doubling constant.
An example of such finite sets is provided by Generalized Arithmetic Progressions (GAP).

\begin{definition}
A \textit{Generalized Arithmetic Progression} (GAP) of dimension $m$ is a finite set of the form
\[L= \{\boldsymbol{v}+\sum_{i=1}^m\ell_i\boldsymbol{v}_i:1\le \ell_i\le N_i \text{ for all $1\le i\le m$}\}\]
where $\boldsymbol{v},\boldsymbol{v}_1,\cdots,\boldsymbol{v}_m\in \RR^n$ and $N_1,\cdots,N_m\in\NN$.
It is called \textit{proper} if its size precisely equals $N_1\cdots N_m$.
\end{definition}
For instance, a lattice gives rise to a proper finite GAP when restricted to a finite region.
By definition, any proper finite GAP has a doubling constant at most $2^m$ regardless of its size.
As a corollary of Theorem \ref{thm:main}, we derive the following result.

\begin{corollary}
Consider a non-degenerate $C^{n,\alpha}$ curve $\Gamma$ inside $\RR^n$ with positive $\alpha$.
Let $L$ be a proper finite GAP, and let $s$ be its minimal separation.
Then for any $\epsilon>0$, we have
\[|\N_{s^n}(\Gamma)\cap L|\lesssim s^{-\epsilon}|L|^{\frac{2}{n(n+1)}}\]
where the implicit constant depends on $\epsilon$, $\|\gamma\|_{C^{n,\alpha}}$ and $c_0$.
\end{corollary}

The corollary, when applied to the special case of a lattice $L$, recovers Bombieri-Pila's inequality as discussed in the Introduction.
\section{Curve-lifting method}
Bombieri and Pila established their estimates, implicitly dealing with a speficit lift of a planar curve into a higher dimensional space.
In this section, we revisit this method and leverage it to derive lattice point estimates on a planar curve.
\subsection{Lift of a planar curve}

\begin{definition}
Let $\Gamma$ be a curve parameterized by $\gamma(t)=(\gamma_1(t),\gamma_2(t))$ for $t\in[0,1]$, and let $\mathcal{M}=\{m_1,\cdots,m_n\}$ be a collection of distinct monomials of two variables.
The \textit{lift} $\Gamma^{\mathcal{M}}$ associated with $\mathcal{M}$ is defined as a curve inside $\RR^n$ parameterized by 
\[\widetilde{\gamma}(t)=(m_1(t),\cdots,m_n(t))\]
for $t\in [0,1]$, where $m_i(t) = m_i(\gamma_1(t),\gamma_2(t))$ for each $i$.
\end{definition}
For instance, a lift of $\Gamma$ with respect to monomials $\{x,y,x^2,xy,y^2\}$ is a curve inside $\RR^5$ parameterized by $\widetilde{\gamma}(t)=(\gamma_1(t),\gamma_2(t),\gamma_1(t)^2,\gamma_1(t)\gamma_2(t),\gamma_2(t)^2)$.

We summarize the main properties of curve-lifting needed later.
\begin{proposition}\label{prop:lift}
Let $\Gamma$ be a curve parameterized by $\gamma(t)=(\gamma_1(t),\gamma_2(t))$ for $t\in[0,1]$, and let $\mathcal{M}=\{m_1,\cdots,m_n\}$ be a collection of distinct monomials of two variables such that $m_1=x$ and $m_2=y$.
\begin{enumerate}
    \item There is a bijection
    \[\Gamma\cap (\tfrac{1}{N}\ZZ)^2 \longleftrightarrow \Gamma^{\mathcal{M}}\cap (\tfrac{1}{N^{d_1}}\ZZ\times\cdots\times \tfrac{1}{N^{d_n}}\ZZ)\]
    where $d_i$ is the degree of $m_i$.
    \item For real numbers $c_0,c_1,\cdots,c_n$, let $C$ be the planar algebraic curve defined by $\sum_{i=1}^n c_im_i(x,y)=c_0$ and $H$ be the hyperplane inside $\RR^n$ defined by $\sum_{i=1}^nc_ix_i=c_0$.
    Then there is a bijection
    \[\Gamma\cap C\longleftrightarrow\Gamma^{\mathcal{M}}\cap H.\]
\end{enumerate}
\end{proposition}\label{prop:basic}

\begin{proof}
Each statement can be checked in a straightforward way as follows.
\begin{enumerate}
    \item It is sufficient to note that for a $\frac{1}{N}$-integer point in $\RR^2$ the corresponding point under the lift has the $i$-th component with denominator divisible by $N^{d_i}$.
    \item In terms of the parameter $t$ of a curve, both sets correspond to $t\in[0,1]$ satisfying $\sum_{i=1}^n c_im_i(\gamma_1(t),\gamma_2(t))=c_0$.
\end{enumerate}
\end{proof}
For a point $p=(x_0,y_o)\in\RR^2$, we will write
\[p^{\mathcal{M}}=(m_1(x_0,y_0),\cdots,m_n(x_0,y_0)).\]
\begin{proposition}\label{prop:ineq}
For points $p,q$ inside a bounded square $[-R,R]^2$, we have 
\[|p^{\mathcal{M}}-q^{\mathcal{M}}|_{\RR^n}\le C(\mathcal{M})|p-q|_{\RR^2}\]
where $C(\mathcal{M})$ only depends on $R$ and the maximal degree of monomials $m_1,\cdots,m_n$.    
\end{proposition}

\begin{proof}
Let $p=(x_1,y_1)$, $q=(x_2,y_2)$ and $m_i=x^{a_i}y^{b_i}$.
    Then we clearly have
    \[|p^{\mathcal{M}}-q^{\mathcal{M}}|_{\RR^n}^2=\sum_{i=1}^n |x_1^{a_i}y_1^{b_i}-x_2^{a_i}y_2^{b_i}|^2.\]
    The triangle inequality implies
    \begin{align*}
        |x_1^{a_i}y_1^{b_i}-x_2^{a_i}y_2^{b_i}|
        &\le |x_1^{a_i}y_1^{b_i}-x_1^{a_i}y_2^{b_i}| + |x_1^{a_i}y_2^{b_i}-x_2^{a_i}y_2^{b_i}|\\
        &=|x_1|^{a_i}|y_1^{b_i}-y_2^{b_i}|+|y_2|^{b_i}|x_1^{a_i}-x_2^{a_i}|.
    \end{align*}
    The simple observation 
    \begin{align*}|X^m-Y^m|&=|X-Y|\cdot|X^{m-1}+\cdots+Y^{m-1}|\\
    &\le m\cdot |X-Y|\cdot(\max\{|X|,|Y|\})^{m-1}\end{align*}
    for real numbers $X,Y$ now implies
    \begin{align*}
        |p^{\mathcal{M}}-q^{\mathcal{M}}|_{\RR^n}^2
        &\le \sum_{i=1}^n\left(b_i R^{a_i+b_i-1}|y_1-y_2|+a_i R^{a_i+b_i-1}|x_1-x_2|\right)^2\\
        &\le \sum_{i=1}^n ((a_i+b_i) R^{a_i+b_i-1}|p-q|_{\RR^2})^2\\
        &\le |p-q|_{\RR^2}^2\sum_{i=1}^n \deg(m_i)^2 R^{2\deg(m_i)-2}.
    \end{align*}
\end{proof}
\subsection{Bombieri-Pila's estimate of lattice points on a planar curve}
In this section, following the approach in \cite{BP}, we present a more general estimate of lattice points on a planar curve by combining the curve-lifting method and determinant method.
\begin{proposition}\label{prop:BP}
Let $\Gamma$ be a compact $C^n$ planar curve, and let $\mathcal{M}=\{m_1,\cdots,m_n\}$ be a collection of $n$ distinct monomials about two variables such that $m_1=x,m_2=y$.
Suppose that there exists a constant $S$ such that $\Gamma$ has at most $S$ intersection points with any algebraic curve of the form $c_1m_1(x,y)+\cdots+c_nm_n(x,y)=c_0$.
Then we have
\[|\Gamma\cap (\tfrac{1}{N}\ZZ)^2|\le C(\Gamma, C)N^{e(\mathcal{M})}\]
where the exponent $e(\mathcal{M})$ is given by
\[e(\mathcal{M})=\frac{2}{n(n+1)}\sum_{i=1}^n \deg m_i.\]
\end{proposition}

\begin{remark}
The assumption in Proposition \ref{prop:BP} is weaker than the one in Theorem \ref{thm:lat}, where the latter allows lattice points to lie in the neighborhood of a curve.
We will later establish as Proposition \ref{prop:nd} that the non-degeneracy implies the finiteness condition regarding intersections with arbitrary hyperplanes.
Huxley achieved a similar upper bound in \cite{H}, using the determinant method.
\end{remark}

\begin{proof}
Let $\Gamma^{\mathcal{M}}$ be the lift of $\Gamma$ with respect to $\mathcal{M}=\{m_1,\cdots,m_n\}$.
We first note that Proposition \ref{prop:lift} implies that
\[|\Gamma\cap (\tfrac{1}{N}\ZZ)^2| = |\Gamma^{\mathcal{M}}\cap (\tfrac{1}{N^{d_1}}\ZZ\times\cdots\times \tfrac{1}{N^{d_n}}\ZZ)|\]
By Proposition \ref{prop:lift} (2), the assumption implies that the lift $\Gamma^{\mathcal{M}}$ has finite intersection with any hyperplane inside $\RR^n$.
Therefore, we can apply the determinant method from \cite{BP} to obtain the inequality
\[|\Gamma^{\mathcal{M}}\cap (\tfrac{1}{N^{d_1}}\ZZ\times\cdots\times \tfrac{1}{N^{d_n}}\ZZ)|\le C(\Gamma^{\mathcal{M}}, C)(N^{d_1}\cdots N^{d_n})^{\frac{2}{n(n+1)}}.\]
\end{proof}

\begin{remark}\label{rmk}
Proposition \ref{prop:BP} suggests a certain finiteness condition for a planar curve with smoothness $C^n$.
Express $n$ in the form $n=2+3+\cdots+s+\Delta n$ for some $s$ where $0\le \Delta n\le s$.
We note a collection $\mathcal{M}$ of $n$ distinct monomials has minimal total degree when it consists of $x^iy^j$, for all pairs $(i,j)$ such that $1\le i+j\le s$, and $\Delta n$ distinct monomials $m_1^{(s)},\cdots,m_{\Delta n}^{(s)}$ of degree $s+1$.
We can then consider the following condition on a planar curve:
\begin{enumerate}
    \item[(*)] A planar curve $\Gamma$ has uniformly bounded intersection with any algebraic curve involving the monomials in $\mathcal{M}$
\end{enumerate}
By Proposition \ref{prop:BP}, the condition (*) implies the following inequality
\[|\Gamma\cap (\tfrac{1}{N})\ZZ)^2|\le C(\Gamma) N^{e(\mathcal{M})}\]
where the exponent is given by
\[e(\mathcal{M}) = \frac{2}{n(n+1)}\left(\frac{(s-1)s(s+1)}{3} + s\cdot \Delta n\right).\]
\end{remark}

\begin{remark}
The curve-lifting method is applicable to curves in any dimension, and the same argument yields an estimate for lattice points on a curve in an arbitrary dimension.
For example, starting with a compact space curve with parameterization $\gamma(t)=(\gamma_1(t),\gamma_2(t),\gamma_3(t))$, we can lift it to a curve inside $\RR^9$ by using monomials $\{x,y,z,x^2,xy,xz,y^2,yz,z^2\}$.
\end{remark}
\section{Lattice points on a planar curve}
\subsection{Proof of Theorem \ref{thm:lat}}
Let $\Gamma$ be a $C^n$ planar curve parameterized by $\Gamma(t)=(\gamma_1(t),\gamma_2(t))$ for $t\in[0,1]$.
For each collection of distinct monomials $\mathcal{M}=\{m_1,\cdots,m_n\}$, we define the function $W^{\mathcal{M}}(\Gamma)(t)$ as
\[W^{\mathcal{M}}(\Gamma)(t)
=\begin{vmatrix} m_1'(t)&m_2'(t)&\cdots&m_n'(t)\\
m_1''(t)&m_2''(t)&\cdots&m_n''(t) \\
\vdots&\vdots&\ddots&\vdots\\
m_1^{(n)}(t)&m_2^{(n)}(t)&\cdots&m_n^{(n)}(t)
\end{vmatrix}\]
where $m_i(t) = m_i(\gamma_1,\gamma_2)(t)$.
This is precisely the Wronskian $W(\Gamma^{\mathcal{M}})(t)$ of the lift $\Gamma^{\mathcal{M}}$ associated with $\mathcal{M}$.

\begin{proof}[Proof of Theorem \ref{thm:lat}]
The first part of Proposition \ref{prop:basic} implies that $\frac{1}{N}$-integer points correspond to points in $\frac{1}{N^{d_1}}\ZZ\times\cdots\times\frac{1}{N^{d_n}}\ZZ$ under the curve-lifting procedure, where $d_i=\deg m_i$.
In particular, the image is $\frac{1}{N}$-separated.
Moreover, Proposition \ref{prop:ineq} ensures that there exists a constant $c$, depending only on the maximal degree and the length of $\Gamma$, such that the corresponding points lie in the $\frac{cd}{N^n}$-neighborhood of the lift $\Gamma^{\mathcal{M}}$.
Therefore, we have the following inequality
\[|\N_{d/N^{n}}(\Gamma)\cap (\tfrac{1}{N}\ZZ)^2|\le |\N_{cd/N^{n}}(\Gamma^{\mathcal{M}})\cap (\tfrac{1}{N^{d_1}}\ZZ\times\cdots\times\tfrac{1}{N^{d_n}}\ZZ)|.\]
By the assumption, the lift $\Gamma^{\mathcal{M}}$ is a non-degenerate $C^{n,\alpha}$ curve inside $\RR^n$.
Therefore, we can apply Theorem \ref{thm:main} to the lattice $\tfrac{1}{N^{d_1}}\ZZ\times\cdots\times\tfrac{1}{N^{d_n}}\ZZ$ restricted to a finite region.
\end{proof}

When $n$ is of the form $n=\frac{1}{2}(s+1)(s+2)-1$, we may consider the collection of all monomials of degree at most $s$, namely $\mathcal{M}_s=\{x^iy^j:1\le i+j\le s\}$.
In this case Theorem \ref{thm:lat} has a simple exponent, and this essentially corresponds to the inequality considered in \cite{BP}.
\begin{corollary}
Let $n$ be an integer of the form $\frac{1}{2}(s+1)(s+2)-1$.
Suppose that $\Gamma$ is a compact $C^{n,\alpha}$ planar curve such that $W^{\mathcal{M}_s}(\Gamma)$ is nonvanishing.
Then we have
\[|\N_{d/N^{n}}(\Gamma)\cap (\tfrac{1}{N}\ZZ)^2|\lesssim N^{\frac{8}{3(s+3)}+\epsilon}.\]
\end{corollary}

\begin{remark}[Relation to Schdit's conjecture]
Schmidt conjectured in \cite{Sc} that $|\Gamma\cap (\tfrac{1}{N}\ZZ)^2|\lesssim N^{1/2+\epsilon}$ holds for a planar graph under a certain condition.
The inequality has been established under stronger assumptions.
For example, Bombieri and Pila proved it for strictly convex $C^{\infty}$ curves in \cite{BP}, and consequently, Pila proved it for $C^{104}$ curves with an additional analytic condition in \cite{P}.

Proposition \ref{prop:BP} with $n=7$ implies $|\Gamma\cap (\tfrac{1}{N}\ZZ)^2|\le C(\Gamma) N^{\frac{1}{2}}$ for any $C^{7}$ curve that has uniformly bounded intersection points with any algebraic curve involving monomials in $\mathcal{M}$, where $\mathcal{M}$ is a collection of distinct monomials consisting of all monomials of degree at most 2 and two degree 3 monomials.
As we will prove in Proposition \ref{prop:nd}, this assumption is satisfied if the Wronskian $W^{\mathcal{M}}(\Gamma)(t)$ is nonvanishing.
\end{remark}

\subsection{Non-degeneracy implies the finiteness condition about intersection with hyperplanes}
The determinant method, as introduced in \cite{BP}, establishes the inequality
\[|\Gamma\cap(\tfrac{1}{N_1}\ZZ\times\cdots\times\cdots\times\tfrac{1}{N_n}\ZZ)|\le C(\Gamma) (N_1\cdots N_n)^{\frac{2}{n(n+1)}}\]
under the assumption that $\Gamma$ is $C^n$ and has uniformly bounded intersection points with hyperplanes.
Pila \cite{P} later proved that nonvanishing of certain Wronskians of the curve, including $W(\Gamma)$, implies the finiteness condition.

However, our observation reveals that the sole nonvanishing of the Wronskian $W(\Gamma)$ is sufficient to yield the same estimate of lattice points on a curve.
The non-degeneracy condition implies the finiteness condition, allowing us to apply the determinant method.
\begin{proposition}\label{prop:nd}
Let $\Gamma$ be a compact, non-degenerate $C^n$ curve inside $\RR^n$.
Then there exists a constant $N(\Gamma)$ such that any hyperplane intersects $\Gamma$ with at most $N(\Gamma)$ points.
\end{proposition}
\begin{proof}
We prove the statement by induction on $n$.
The base case for $n=2$ is clear.

Suppose that the statement holds for $n-1$, and let $\Gamma$ be a compact, non-degenerate $C^n$ curve inside $\RR^n$.
Since we can partition $\Gamma$ into finitely many graphs, we may assume that $\Gamma$ itself is a graph.
Further, we may assume that the graph $\Gamma$ is of the form $\Gamma=\{\gamma(t)=(t,f_2(t),\cdots,f_n(t)) : t\in I\}$.

Let us define a $C^{n-1}$ graph $\Gamma'$ inside $\RR^{n-1}$ using the parameterization:
\[\Gamma'=\{(f_2'(t),\cdots,f_n'(t)): t\in I\}.\]
This graph is non-degenerate since $\Gamma$ is non-degenerate. 
We make the following observation.

\vspace{3pt}
\textbf{Observation.}
\textit{If a hyperplane $H$ inside $\RR^n$ intersects the graph $\Gamma$ with $m$ points, then there exists a hyperplane $H'$ inside $\RR^{n-1}$ that intersects the graph $\Gamma'$ with at least $m-1$ points.}
\vspace{3pt}

To prove the observation, consider a hyperplane $H:a_1x_1+\cdots+a_nx_n=a_0$ inside $\RR^n$ intersecting with $\Gamma$ at points $\gamma(t_1),\cdots,\gamma(t_m)$ for $t_1,\cdots,t_m\in I$.
This implies
\[a_0+a_1t_i+a_2f_2(t_i)+\cdots+a_nf_n(t_i)=0\]
for each $i$.
By the Mean Value Theorem, there exist points $t_1',\cdots,t_{m-1}'\in I$ within the range $t_i<t_i'<t_{i+1}$ for $1\le i\le m-1$ such that
\[a_1+a_2f_2'(t_i')+\cdots+a_nf_n'(t_i')=0\]
for each $i$.
Consequently, the hyperplane $H':a_1+a_2x_2+\cdots+a_nx_n=0$ inside $\RR^{n-1}$ intersects $\Gamma'$ with at least $m-1$ points.

The proof concludes by the inductive hypothesis.
There exists a constant $N$ such that $\Gamma'$ intersects any hyperplane inside $\RR^{n-1}$ with at most $N$ points.
According to the observation, the graph $\Gamma$ meets any hyperplane inside $\RR^n$ with at most $N+1$ points.
\end{proof}

\section*{Acknowledgement}
We would like to thank Larry Guth for his suggestion to study the application of $\ell^2$ decoupling to lattice points and generously sharing his insights.
A part of the research was conducted while the author was a participant of the Summer Program in Undergraduate Research (SPUR) of the MIT Mathematics Department.
We would like to thank David Jerison and Ankur Moitra for organizing the program and Feng Gui for his support as a mentor.
We would also like to thank Yuqiu Fu for helpful discussions.

We would like to thank the anonymous referee for thoroughly reading the draft and providing helpful suggestions.
\appendix


\begin{thebibliography}{0}
\bibitem{BP} E. Bombieri and J. Pila, The number of integral points on arcs and ovals, {\it Duke Math. J.} {\bf 59}(2) (1989) 337--357.

\bibitem{BD} J.Bourgain and C. Demeter, The proof of the {$l^2$} decoupling conjecture, {\it Ann. of Math. (2)} {\bf 182}(1) (2015) 351--389

\bibitem{BDGuth} J. Bourgain, C. Demeter and L. Guth, Proof of the main conjecture in {V}inogradov's mean value theorem for degrees higher than three, {\it Ann. of Math. (2)} {\bf 184}(2) (2016) 633--682


\bibitem{GLYZ} S. Guo and Z.K. Li and P.L. Yung and P. Zorin-Kranich, A short proof of {$\ell^2$} decoupling for the moment curve, {\it Amer. J. Math.} {\bf 143}(6) (2021) 1983--1998

\bibitem{H} M.N. Huxley, The integer points in a plane curve, {\it Funct. Approx. Comment. Math.} {\bf 37}(1) (2007) 213--231

\bibitem{J} V. Jarn\'{\i}k, \"{U}ber die {G}itterpunkte auf konvexen {K}urven, {\it Math. Z.} {\bf 24}(1) (1926) 500--518

\bibitem{P} J. Pila, Geometric postulation of a smooth function and the number of rational points, {\it Duke Math. J.} {\bf 63}(2) (1991) 449--463

\bibitem{PS} M. Pramanik and A. Seeger, {$L^p$} regularity of averages over curves and bounds for associated maximal operators, {\it Amer. J. Math.} {\bf 129}(1) (2007) 61--103

\bibitem{Sc} W. M. Schmidt, Integer points on curves and surfaces, {\it Monatsh. Math.} {\bf 99}(1) (1985) 45--72

\bibitem{Sw} H. P. F. Swinnerton-Dyer, The number of lattice points on a convex curve, {\it J. Number Theory} {\bf 6} (1974) 128--135.

\end{thebibliography}
\end{document}